\documentclass[12pt]{article}
\usepackage{amssymb}
\usepackage{amsmath}
\usepackage{amsthm}
\usepackage{enumerate}
\usepackage{graphicx}
\usepackage{amsthm}
\usepackage{color}
\usepackage{graphicx}
\usepackage{amsfonts}
\def\wT{\widehat{T}}

\newcommand{\ignore}[1]{}

\def\Bin{{\rm Bin}}

\def\KS2{{\cal KS}2}

\newcommand{\brac}[1]{\left(#1\right)}
\newcommand{\bfrac}[2]{\brac{\frac{#1}{#2}}}

\newcommand{\set}[1]{\left\{#1\right\}}

\def\cC{{\cal C}}
\def\cE{{\cal E}}

\def\cE{{\cal E}}

\def\a{\alpha}
\def\b{\beta}
\def\d{\delta}
\def\D{\Delta}
\def\f{\phi}

\def\G{\Gamma}
\def\k{\kappa}

\def\l{\lambda}

\def\m{\mu}
\def\n{\nu}

\def\r{\rho}

\def\om{\omega}

\def\whp{w.h.p.}

\newtheoremstyle{plain}%
    {8pt plus2pt minus4pt}%
    {8pt plus2pt minus4pt}%
    {\itshape}%
    {}%
    {\bfseries\scshape}%
    {}%
    {6pt}% Space after theorem head
    {}%

\newtheoremstyle{remark}%
    {8pt plus2pt minus4pt}%
    {8pt plus2pt minus4pt}%
    {\upshape}% Body font
    {}%
    {\bfseries\scshape}%
    {}%
    {6pt}% Space after theorem head
    {}%

\theoremstyle{plain}
\newtheorem{theorem}{Theorem}%[section]
\newtheorem{corollary}[theorem]{Corollary}

\newtheorem{lemma}[theorem]{Lemma}

\theoremstyle{remark}

\newcommand{\beq}[1]{\begin{equation}\label{#1}}
\newcommand{\eeq}{\end{equation}}

\def\cC{\mathcal{C}}
\begin{document}

\title{Rainbow Connection of Random Regular Graphs}

\author{Andrzej Dudek
\thanks{Department of Mathematics, Western Michigan University, Kalamazoo, MI 49008. E-mail: \texttt{andrzej.dudek@wmich.edu}.
Research supported in part by Simons Foundation Grant \#244712.}
\and Alan Frieze\thanks{Department of Mathematical Sciences,
Carnegie Mellon University, Pittsburgh, PA 15213. E-mail: \texttt{alan@random.math.cmu.edu}.
Research supported in part by CCF1013110.}
\and Charalampos E. Tsourakakis
\thanks{Harvard School of Engineering and Applied Sciences, Cambridge, MA 02138.
E-mail: \texttt{babis@seas.harvard.edu}.}}

\maketitle

\begin{abstract} 
An edge colored graph $G$ is rainbow edge connected if any two vertices are connected
by a path whose edges have distinct colors. The rainbow connection of a connected
graph $G$, denoted by $rc(G)$, is the smallest number of colors that are needed in
order to make $G$ rainbow connected. 

In this work we study the rainbow connection of the random 
$r$-regular graph $G=G(n,r)$ of order~$n$, where  $r\ge 4$ is a constant.
We prove that with probability tending to one as $n$ goes to infinity 
the rainbow connection of $G$ satisfies 
$rc(G)=O(\log n)$, which is best possible up to a hidden constant.
\end{abstract} 

\section{Introduction} 

Connectivity is a fundamental graph theoretic property. Recently, the  
concept of rainbow connection was introduced by 
Chartrand, Johns, McKeon and Zhang in \cite{chartrand}. 
We say that a set of edges is {\em rainbow colored}
if its every member has a distinct color. An edge colored graph $G$ is {\em rainbow edge connected} 
if any two vertices are connected by a 
rainbow colored path. Furthermore, the {\em rainbow connection} $rc(G)$ of a connected graph 
$G$ is the smallest number of colors that are needed 
in order to make $G$ rainbow edge connected. 

Notice, that by definition a rainbow edge connected 
graph is also connected. Moreover,
any connected graph has a trivial edge coloring that 
makes it rainbow edge connected, since 
one may color the edges of a given spanning tree with distinct colors. 
Other basic facts established in \cite{chartrand} are 
that $rc(G)=1$ if and only if $G$ is a 
clique and $rc(G)=|V(G)|-1$ if and only if $G$ is a tree. 
Besides its theoretical interest, rainbow connection 
is also of interest in applied settings, such 
as securing sensitive information transfer and networking 
(see, e.g., \cite{chakraborty, lisun}).
For instance, consider the following setting in networking 
\cite{chakraborty}: we want to route messages in 
a cellular network such that each link on the route between
 two vertices is assigned with a distinct channel. 
Then, the minimum number of channels to 
use is equal to the rainbow connection of the underlying network.

Caro, Lev, Roditty, Tuza and Yuster  
\cite{caro} prove that for a connected 
graph $G$ with $n$ vertices and minimum degree $\delta$, 
the rainbow connection satisfies
$rc(G)\leq \frac{\log{\delta}}{\delta}n(1+f(\delta))$, where $f(\delta)$ tends to zero 
as $\delta$ increases. The following simpler bound was also proved in \cite{caro},
$rc(G) \leq n \frac{4\log{n}+3}{\delta}$. 
Krivelevich and Yuster \cite{krivelevichyuster} 
removed the logarithmic factor
from the  upper bound in \cite{caro}. Specifically they proved that
$rc(G) \leq \frac{20n}{\delta}$. Chandran, Das, Rajendraprasad
and Varma \cite{CDRV} improved this upper bound to $\frac{3n}{\d+1}+3$,
which is close to best possible.

As pointed out in \cite{caro} 
the random graph setting poses several intriguing questions. 
Specifically, let $G=G(n,p)$ denote the binomial random graph on $n$ 
vertices with edge probability $p$.
Caro, Lev, Roditty, Tuza and Yuster \cite{caro} proved that $p=\sqrt{\log{n}/n}$ is the sharp threshold for 
the property $rc(G)\leq 2$. This was sharpened to a hitting time result
by Heckel and Riordan \cite{HR}.
He and Liang \cite{heliang} studied further the rainbow connection of random graphs. 
Specifically, they obtain a threshold for the property $rc(G) \leq d$
where $d$ is constant. 
Frieze and Tsourakakis \cite{gnprainbow} studied  
the rainbow connection of $G=G(n,p)$ 
at the connectivity threshold $p=\frac{\log{n}+{\om}}{n}$ 
where $\om\to\infty$ and $\om=o(\log{n})$. They showed that  \whp\footnote{An event $\cE_n$
occurs {\em with high probability}, or \whp\ for brevity, if
$\lim_{n\rightarrow\infty}\Pr(\cE_n)=1$.} 
$rc(G)$ is asymptotically equal to $\max\set{diam(G), Z_1(G)}$, where $Z_1$ is the number of vertices of 
degree one. 

For further results and references we refer 
the interested reader to the recent survey of Li, Shi and Sun \cite{lisun}. 

In this paper we study the rainbow connection of
the random $r$-regular graph $G(n,r)$ of order $n$, where $r\geq 4$ is a constant and
$n\to \infty$. It was shown in Basavaraju, Chandran, Rajendraprasad, and Ramaswamy~\cite{BCRR} that for any bridgeless graph $G$,
$rc(G)\leq \r(\r+2)$, where $\r$ is the radius of $G=(V,E)$, i.e., 
$\min_{x\in V}\max_{y\in V} dist(x,y)$. Since the radius of $G(n,r)$ is $O(\log n)$ \whp,  we see that \cite{BCRR} implies that
$rc(G(n,r))=O(\log^2n)$ \whp The following theorem gives an improvement on this for $r\geq 4$.  

\begin{theorem}\label{thrm:mainthrm}
Let $r\ge 4$ be a constant. Then, \whp\
$rc(G(n,r))=O(\log n)$.
\end{theorem} 

\noindent The rainbow connection of any graph~$G$ is at least as large 
as its diameter. The diameter of $G(n,r)$ is \whp\ 
asymptotically $\log_{r-1}{n}$ and so the above theorem is best possible,
up to a (hidden) constant factor.

We conjecture that Theorem \ref{thrm:mainthrm} can be extended to include $r=3$.
Unfortunately, the approach taken in this paper does not seem to work in this case.

\section{Proof of Theorem~\ref{thrm:mainthrm}} 
\label{sec:regular} 

%\subsection{Chernoff bounds}
%We will use the following bounds on the tails of the binomial distributuion
%$B(n,p)$:
%\begin{align*}
%&\Pr(|B(n,p)-np|\geq \e np)\leq 2e^{-\e^2np/3},\qquad 0\leq \a\leq 1.\\
%&\Pr(B_{n,p}\geq \a np)\leq \bfrac{e}{\a}^{\a np},\qquad \a\geq 1.
%\end{align*}

\subsection{Outline of strategy}
Let $G=G(n,r)$, $r\ge 4$. Define 
\begin{equation}\label{eq:kr}
k_r= \log_{r-1}(K_1\log n),
\end{equation}
where $K_1$ will be a sufficiently large absolute constant. Recall that the {\em distance between two vertices} in $G$ is the number of edges in a shortest path connecting them and the {\em distance between two edges} in $G$ is the number of vertices in a shortest path between them. (Hence, both adjacent vertices and incident edges have distance 1.)

For each vertex $x$ let $T_x$ be the subgraph of $G$ induced by the vertices within distance
$k_r$ of $x$.
We will see (due to Lemma~\ref{density})
that \whp, $T_x$ is a tree for most $x$ and that
for all $x$, $T_x$ contains at most one cycle. We say that $x$ is 
{\em tree-like} if $T_x$ is a tree. In which case we denote by $L_x$ the
leaves of $T_x$. Moreover, if $u\in L_x$, then we denote the path
from $u$ to $x$ by $P(u,x)$. 

We will randomly color $G$ in such a way that the edges of every path $P(u,x)$  is rainbow
colored for all $x$. This is how we do it. We order the edges of $G$ in some
arbitrary manner as $e_1,e_2,\ldots,e_m$, where $m=rn/2$. There will be a set
of $q = \lceil K_1^2 r\log n\rceil$ 
colors available. Then, in the order $i=1,2,\ldots,m$
we randomly color $e_i$. We choose this color uniformly
from the set of colors not used by those $e_j,j<i$ which are within 
distance $k_r$ of $e_i$. Note that the number of edges within distance $k_r$ of $e_i$ is at most
\begin{equation}\label{eq:edges_k_r}
2\left( (r-1)  +(r-1)^2 + \dots + (r-1)^{\lfloor k_r \rfloor-1} \right) \le (r-1)^{k_r} = K_1 \log n.
\end{equation}
So for $K_1$ sufficiently large we always have many colors that can be used for $e_i$.
Clearly, in such a coloring, the edges of a path $P(u,x)$ are rainbow colored.

Now consider a fixed pair of tree-like vertices $x,y$. We will show (using Corollary~\ref{cor1}) that one can find
a partial 1-1 mapping $f=f_{x,y}$ between $L_x$ and $L_y$ such that if
$u\in L_x$ is in the domain $D_{x,y}$ 
of $f$ then $P(u,x)$ and $P(f(u),y)$ do not share
any colors. The domain $D_{x,y}$ 
of $f$ is guaranteed to be of size at least $K_2\log n$, where $K_2=K_1/10$.

Having identified $f_{x,y},D_{x,y}$ we then search for a rainbow path
joining $u\in D_{x,y}$ to $f(u)$. 
%There are at least $K_2\log n$ choices for $u$ and this means that we will succeed \whp 
To join $u$ to $f(u)$
we continue to grow the trees $T_x,T_y$ until there are 
$n^{1/20}$ leaves. Let the new larger trees be denoted by $\wT_x,\wT_y$, respectively. As we
grow them, we are careful to prune away edges where the edge to root path
is not rainbow. We do the same with $T_y$ and here make sure that edge
to root paths are rainbow with respect to corresponding $T_x$ paths. 
We then construct at least $n^{1/21}$ vertex disjoint paths $Q_1,Q_2,\ldots,$ from the leaves
of $\wT_x$ to the leaves of $\wT_y$. We then argue that \whp\ one of these
paths is rainbow colored and that the colors used are disjoint from the 
colors used on $P(u,x)$ and $P(f(u),y)$.

We then finish the proof by dealing with non tree-like vertices in 
Section \ref{nontree}.

\subsection{Coloring lemmata}
In this section we prove some auxiliary results about rainbow colorings of $d$-ary trees.

Recall that a {\em complete $d$-ary tree $T$} is a rooted tree in
which each non-leaf vertex has exactly $d$ children. The {\em depth}
of an edge is the number of vertices in the path connecting the root
to the edge. The set of all edges at a given depth is called a {\em
  level} of the tree. The {\em height} of a tree is the distance from
the root to the deepest vertices in the tree (i.e. the leaves). Denote by $L(T)$ the set of leaves and for $v\in L(T)$ let $P(v,T)$ be the path from the root
of $T$ to $v$ in $T$.

\begin{lemma}\label{lemcol}
Let $T_1,T_2$ be two vertex disjoint \emph{rainbow} copies of  
the complete $d$-ary tree with $\ell$ levels, where 
$d\geq 3$. Let $T_i$ be rooted at $x_i$, $L_i=L(T_i)$ for $i=1,2$, and
\[
m(T_1,T_2)=\left| \{ (v,w)\in L_1\times L_2 : P(v,T_1) \cup P(w,T_2)\text{ is rainbow} \}\right|.
\]
Let 
$$\k_{\ell,d}=\min_{T_1,T_2}\set{m(T_1,T_2)}.$$
Then, 
\begin{equation}\label{eq:col_lem1}
\k_{\ell,d}\geq d^{2\ell}/4.
\end{equation}
\end{lemma}

\begin{proof} We prove that 
\[
\k_{\ell,d}\geq \brac{1-\sum_{i=1}^\ell \frac{i}{d^{i}}} d^{2\ell}\geq d^{2\ell}/4.
\]
We prove this by induction on $\ell$. If $\ell=1$, then clearly
\[
\k_{1,d}=d(d-1).
\]
Suppose that \eqref{eq:col_lem1} holds for an $\ell \ge 2$.

Let $T_1,T_2$ be rainbow trees of height $\ell+1$. Moreover, let $T_1' = T_1\setminus L(T_1)$ and 
$T_2' = T_2\setminus L(T_2)$. 
We show that 
\begin{equation}\label{eq:col_lem1:m}
m(T_1,T_2) \ge d^2 \cdot m(T_1',T_2') - (\ell+1)d^{\ell+1}.
\end{equation}
Each $(v',w')\in L_1'\times L_2'$ gives rise to $d^2$ pairs of leaves 
$(v,w)\in L_1\times L_2$, where $v'$ is the parent of $v$ and $w'$ is the parent of $w$. 
Hence, the term $d^2 \cdot m(T_1',T_2')$ accounts for the pairs 
$(v,w)$, where $P_{v',T_1'}\cup P_{w',T_2'}$ is rainbow. 
We need to subtract off those  pairs for which $P_{v,T_1}\cup P_{w,T_2}$ is not rainbow. 
Suppose that this number is $\n$.
Let $v\in L(T_1)$ and let $v'$ be its parent, and let
$c$ be the color of the edge $(v,v')$. 
Then  $P_{v,T_1}\cup P_{w,T_2}$ is rainbow
unless $c$ is the color of some edge of $P_{w,T_2}$. 
Now let $\n(c)$ denote the number of root to leaf paths in $T_2$ that 
contain an edge color $c$. Thus,
\[
\n \leq \sum_{c} \n(c),
\]
where the summation is taken over all colors $c$ that appear in edges of $T_1$ adjacent to leaves. We bound this sum trivially, by summing over all colors in $T_2$ (i.e., over all edges in $T_2$, since $T_2$ is rainbow). Note that if the depth of the edge colored $c$ in $T_2$ is $i$, then $\n(c) \le d^{\ell+1-i}$.
Thus, summing over edges of $T_2$ gives us
\[
\sum_c\n(c)\leq\sum_{i=1}^{\ell+1}d^{\ell+1-i} \cdot d^i=(\ell+1)d^{\ell+1},
\]
and consequently \eqref{eq:col_lem1:m} holds. Thus, by induction (applied to $T_1'$ and $T_2'$)
\begin{align*}
m(T_1,T_2) &\ge d^2 \cdot m(T_1',T_2') - (\ell+1)d^{\ell+1}\\
&\ge d^2 \brac{1-\sum_{i=1}^\ell \frac{i}{d^{i}}}d^{2\ell} - (\ell+1)d^{\ell+1}\\
&\ge \brac{1-\sum_{i=1}^{\ell+1} \frac{i}{d^{i}}}d^{2(\ell+1)},
\end{align*}
as required.

\end{proof}

In the proof of Theorem~\ref{thrm:mainthrm} we will need a stronger version of the above lemma.

\begin{lemma}\label{lemcol1}
Let $T_1,T_2$ be two vertex disjoint edge colored copies of  
the complete $d$-ary tree with  $L$ levels, where 
$d\geq 3$. For $i=1,2$, let $T_i$ be rooted at $x_i$ and
suppose that edges $e,f$ of $T_i$ have a different color whenever
the distance between $e$ and $f$ in $T_i$ is at most $L$. Let
$\k_{\ell,d}$ be as defined in Lemma \ref{lemcol}. Then
\[
\k_{L,d}\geq   \brac{1-\frac{L^2}{d^{\lfloor L/2 \rfloor}}-\sum_{i=1}^{\lfloor L/2 \rfloor} \frac{i}{d^{i}}}d^{2L}.
\]
\end{lemma}
\begin{proof}
Let $T_i^\ell$ be the subtree of $T_i$ spanned by the first $\ell$ levels, where $1\le \ell \le L$ and $i=1,2$.
We show by induction on $\ell$ that
\begin{equation}\label{eq:col_lem2}
m(T_1^\ell, T_2^{\ell}) \ge \brac{1-\frac{\ell^2}{d^{\lfloor L/2 \rfloor}}-\sum_{i=1}^{\lfloor L/2 \rfloor} \frac{i}{d^{i}}}d^{2\ell}.
\end{equation}
Observe first that Lemma \ref{lemcol} implies \eqref{eq:col_lem2} for $1\le \ell \le \lfloor L/2 \rfloor -1$, since in this case $T_1^\ell$ and $T_2^\ell$ must be rainbow.

Suppose that $\lfloor L/2 \rfloor \leq\ell<L$ and consider the case where $T_1,T_2$ have height $\ell+1$. Following the argument of Lemma \ref{lemcol}
we observe that color $c$ can be the color of at most $d^{\ell +1 - \lfloor L/2 \rfloor}$ leaf edges of $T_1$. This is because for two leaf edges to 
have the same color, their common ancestor must be at distance (from the root) at most 
$\ell-\lfloor L/2 \rfloor$. Therefore, 
\begin{align*}
m(T_1^{\ell+1}, T_2^{\ell+1}) &\geq d^{2}\cdot m(T_1^\ell, T_2^\ell)  -d^{\ell+1-\lfloor L/2 \rfloor} \sum_c\n(c) \\
&\geq d^{2}\cdot m(T_1^\ell, T_2^\ell)  -d^{\ell+1-\lfloor L/2 \rfloor }(\ell+1)d^{\ell+1}\\
&= d^{2}\cdot m(T_1^\ell, T_2^\ell)  -(\ell+1)d^{2(\ell+1)-\lfloor L/2 \rfloor }.
\end{align*}
Thus, by induction
\begin{align*}
m(T_1^{\ell+1}, T_2^{\ell+1}) &\ge d^{2}\brac{1-\frac{\ell^2}{d^{\lfloor L/2 \rfloor}}-\sum_{i=1}^{\lfloor L/2 \rfloor} \frac{i}{d^{i}}}d^{2\ell}  -(\ell+1)d^{2(\ell+1)-\lfloor L/2 \rfloor }\\
&= \brac{1-\frac{\ell^2 + \ell+1}{d^{\lfloor L/2 \rfloor}}-\sum_{i=1}^{\lfloor L/2 \rfloor} \frac{i}{d^{i}}}d^{2(\ell+1)}\\
&\ge \brac{1-\frac{(\ell+1)^2}{d^{\lfloor L/2 \rfloor}}-\sum_{i=1}^{\lfloor L/2 \rfloor} \frac{i}{d^{i}}}d^{2(\ell+1)}
\end{align*}
yielding~\eqref{eq:col_lem2} and consequently the statement of the lemma.
\end{proof}

\begin{corollary}\label{cor1}
Let $T_1,T_2$ be as in Lemma \ref{lemcol1}, except that
the root degrees are $d+1$ instead of $d$. If $d\geq 3$ and $L$ is
sufficiently large, then 
there exist $S_i\subseteq L_i, i=1,2$ and a bijection $f:S_1\to S_2$ such that
\begin{enumerate}[(a)]
\item $|S_i|\geq d^L/10$, and
\item $x\in S_1$ implies that $P_{x,T_1}\cup P_{f(x),T_2}$ is rainbow.
\end{enumerate}
\end{corollary}

\begin{proof}
To deal with the root degrees being $d+1$ we simply ignore one of the 
subtrees of each of the roots.
Then note that if $d\geq 3$ then 
\[
1-\frac{L^2}{d^{\lfloor L/2 \rfloor}}-\sum_{i=1}^{\lfloor L/2 \rfloor} \frac{i}{d^{i}}
\ge 1-\frac{L^2}{d^{\lfloor L/2 \rfloor}}-\sum_{i=1}^{\infty} \frac{i}{d^{i}}
= 1-\frac{L^2}{d^{\lfloor L/2 \rfloor}}-\frac{d}{(d-1)^2} \ge \frac{1}{5}
\]
for $L$ sufficiently large.
Now we choose $S_1,S_2$ in a greedy manner. Having chosen a matching 
$(x_i,y_i=f(x_i))\in L_1\times L_2$, $i=1,2,\ldots,p$, and $p < d^L/10$, there will still be at
least
$d^{2L}/5-2pd^L>0$ pairs in $m(T_1,T_2)$ that can be added to the matching.
\end{proof}

\subsection{Configuration model}\label{sec:configuration}
We will use the configuration model of Bollob\'as \cite{b1} in our proofs 
(see, e.g.,  \cite{bollobas, JLR, wormald} for details).
Let $W=[2m=rn]$ be our set
of {\em configuration points} and let $W_i=[(i-1)r+1,ir]$,
$i\in [n]$, partition $W$. The function $\f:W\to[n]$ is defined by
$w\in W_{\f(w)}$. Given a
pairing $F$ (i.e. a partition of $W$ into $m$ pairs) we obtain a
(multi-)graph $G_F$ with vertex set $[n]$ and an edge $(\f(u),\f(v))$ for each
$\{u,v\}\in F$. Choosing a pairing $F$ uniformly at random from
among all possible pairings $\Omega_W$ of the points of $W$ produces a random
(multi-)graph $G_F$. 
Each $r$-regular simple graph $G$ on vertex set $[n]$ is 
equally likely to be generated as $G_F$.
Here simple means without loops or multiple edges. 
Furthermore, if $r$ is a constant, then $G_F$ is simple with a 
probability bounded below by a positive value independent of $n$.
Therefore, any event that occurs \whp\ in $G_F$ will also occur \whp\ in $G(n,r)$.

\subsection{Density of small sets}
Here we show that \whp\  almost every subgraph of a random regular
graph induced by the vertices within a certain small distance is a tree.
Let
\begin{equation}\label{eq:ell1}
t_0= \frac{1}{10}\log_{r-1}n.
\end{equation}
\begin{lemma}\label{density}
Let $k_r$ and $t_0$ be defined in \eqref{eq:kr} and \eqref{eq:ell1}. Then, \whp\ in $G(n,r)$ 
\begin{enumerate}[(a)]
\item\label{lem:a} no set of $s\leq t_0$ vertices contains more than $s$ edges, and
\item\label{lem:b} there are at most $\log^{O(1)}n$ vertices that are within distance $k_r$ of a cycle of length at most~$k_r$.
\end{enumerate}
\end{lemma}

\begin{proof}
We use the configuration model described in Section~\ref{sec:configuration}. It follows directly from the definition of this model that the probability that a given set of $k$ disjoint pairs in $W$ is contained in a random configuration is given by
\[
p_k = \frac{1}{(rn-1)(rn-3)\dots(rn-2k+1)} \le \frac{1}{(rn - 2k)^k} \le \frac{1}{r^k (n-k)^k}.
\]
Thus, in order to prove \eqref{lem:a} we bound: 
\begin{align*}
\Pr(\exists S\subseteq [n],|S|\leq t_0,e[S]\geq |S|+1)
&\leq \sum_{s=3}^{\lfloor t_0 \rfloor}\binom{n}{s}\binom{\binom{s}{2}}{s+1} r^{2(s+1)} p_{s+1}\\
&\leq \sum_{s=3}^{\lfloor t_0 \rfloor} \left( \frac{en}{s} \right)^s \left( \frac{es}{2} \right)^{s+1} \left(\frac{r}{n-(s+1)}\right)^{s+1}\\
&\leq  \frac{et_0}{2} \cdot \frac{r}{n-(t_0+1)} \cdot \sum_{s=3}^{\lfloor t_0 \rfloor} \left( \frac{en}{s} \cdot \frac{es}{2}  \cdot \frac{r}{n-(s+1)}\right)^{s}\\
&\leq  \frac{et_0}{2} \cdot \frac{r}{n-(t_0+1)} \cdot \sum_{s=3}^{\lfloor t_0 \rfloor} \left( e^2 r \right)^{s}\\
&\leq  \frac{et_0}{2} \cdot \frac{r}{n-(t_0+1)} \cdot t_0 \cdot \left( e^2 r \right)^{t_0}\\
&\leq  \frac{e r t_0^2}{2(n-(t_0+1))} \cdot n^{\frac{\log_{r-1} (e^2 r)}{10}} = o(1),
\end{align*}
as required.

We prove \eqref{lem:b} in a similar manner. The expected number of vertices within $k_r$ of a cycle of length at most $k_r$ can be bounded from above by
\begin{align*}
\sum_{\ell=0}^{\lfloor k_r \rfloor}\binom{n}{\ell} \sum_{k=3}^{\lfloor k_r \rfloor}\binom{n}{k}\frac{(k-1)!}{2}
r^{2(k+\ell)} p_{k+\ell}
& \leq \sum_{\ell=0}^{\lfloor k_r \rfloor} \sum_{k=3}^{\lfloor k_r \rfloor} n^{k+\ell} \left( \frac{r}{n-(k+\ell)} \right)^{k+\ell}\\
& \leq \sum_{\ell=0}^{\lfloor k_r \rfloor} \sum_{k=3}^{\lfloor k_r \rfloor} \left( 2r \right)^{k+\ell}\\
& \leq k_r^2 (2r)^{2k_r} = \log^{O(1)} n.
\end{align*}
Now \eqref{lem:b} follows from the Markov inequality.
\end{proof}

\subsection{Chernoff bounds}
In the next section we will use the following bounds on the tails of the binomial distribution
$\Bin(n,p)$ (for details, see, e.g., \cite{JLR}):
\begin{align}
&\Pr(\Bin(n,p)\leq \a np)\leq e^{-(1-\a)^2np/2},\quad 0\leq \a\leq 1, \label{chernoff_lower}\\
&\Pr(\Bin(n,p)\geq \a np)\leq \bfrac{e}{\a}^{\a np},\quad \a\geq 1. \label{chernoff_upper}
\end{align}

\subsection{Coloring the edges}
We now consider the problem of coloring the edges of $G=G(n,r)$. 
Let $H$ denote the line graph of 
$G$ and let $\G=H^{k_r}$ denote the 
graph with the same vertex set as $H$ and an edge between
vertices $e,f$ of $\G$ if there there is a path of length 
at most $k_r$ between $e$ and $f$ in $H$. 
Due to \eqref{eq:edges_k_r} the maximum degree $\D(\Gamma)$ satisfies
\beq{maxH}
\D(\Gamma)\leq K_1\log n.
\eeq

We will
construct a proper coloring
of $\G$ using 
\beq{qcol}
q= \lceil K_1^2r\log n \rceil
\eeq
colors. Let $e_1,e_2,\ldots,e_m$ with $m=rn/2$ be an 
arbitrary ordering of the vertices of $\G$. For 
$i=1,2,\ldots,m$, color $e_i$ with a random color, chosen 
uniformly from the set of colors not currently
appearing on any neighbor in $\G$. At this point only 
$e_1,e_2,\ldots,e_{i-1}$ will have been colored.

Suppose then that we color the edges of $G$ using the 
above method. Fix a pair of vertices $x,y$ of $G$.
\subsubsection{Tree-like and disjoint}
Assume first that  $T_x,T_y$ are vertex disjoint and that $x,y$ are both tree-like.
We see immediately, that $T_x,T_y$ fit the conditions of Corollary \ref{cor1}
with $d=r-1$ and $L=k_r$. 
Let $S_x\subseteq L(T_x)$, $S_y\subseteq L(T_y)$, $f:S_x\to S_y$ be the sets
and function promised by Corollary \ref{cor1}. 
Note that $|S_x|,|S_y|\geq K_2\log n$, where $K_2=K_1/10$.

In the analysis below we will expose the pairings in the configuration
as we need to. Thus an unpaired point of $W$ will always be paired to a 
random unpaired point in~$W$.

We now define a sequence
$A_0=S_x,A_1,\ldots,A_{t_0}$,
where $t_0$ defined as in \eqref{eq:ell1}. They are defined so that $T_x\cup A_{\leq t}$
spans a tree $T_{x,t}$ where $A_{\leq t}=\bigcup_{j\leq t}A_j$.
Given $A_1,A_2,\ldots,A_i=\set{v_1,v_2,\ldots,v_p}$
we go through $A_i$ in the order $v_1,v_2,\ldots,v_p$ and construct
$A_{i+1}$. Initially, $A_{i+1}=\emptyset$. When dealing with $v_j$ we
add $w$ to $A_{i+1}$ if:
\begin{enumerate}[(a)]
\item $w$ is a neighbor of $v_j$;
\item\label{A:b} $w\notin T_x\cup T_y\cup A_{\leq i+1}$
(we include $A_{i+1}$ in the union because we do not want to add $w$ to $A_{i+1}$
twice);
\item\label{A:c} If the path $P(v_j,x)$ from $v_j$ to $x$ in $T_{x,i}$ goes 
through $v\in S_x$ then the set of edges $E(w)$ is rainbow colored, 
where $E(w)$ comprises the edges in $P(x,v_j)+(v_j,w)$ and the
edges in the path $P(f(v),y)$ in $T_y$ from $y$ to $f(v)$.
\end{enumerate}
We do not add neighbors of $v_j$ to $A_{i+1}$ if ever one of \eqref{A:b} or \eqref{A:c} fails. We prove next that 
\beq{grow}
\Pr\left(|A_{i+1}|\leq (r-1.1)|A_i|\;\big{|}\, K_2\log n\leq |A_i|\leq n^{2/3}\right)=o(n^{-3}).
\eeq

Let $X_{\ref{A:b}}$ and $X_{\ref{A:c}}$ be the number of vertices lost because of case \eqref{A:b} and \eqref{A:c}, respectively.
Observe that
\begin{equation}\label{eq:Abounds}
(r-1)|A_i|-X_{\ref{A:b}}-X_{\ref{A:c}} \le |A_{i+1}| \le (r-1)|A_i|
\end{equation}
First we show that  $X_{\ref{A:b}}$ is dominated by the binomial random variable 
\[
Y_{\ref{A:b}} \sim (r-1)\Bin\left((r-1)|A_i|,\frac{r|A_i|}{rn/2-rn^{2/3}}\right)
\]
conditioning on $K_2\log n\leq |A_i|\leq n^{2/3}$.
This is because we have to pair up $(r-1)|A_i|$ points and each point
has a probability less than $\frac{r|A_i|}{rn/2-rn^{2/3}}$ of being
paired with a point in $A_i$. (It cannot be paired with a point in
$A_{\leq i-1}$ because these points are already paired up at this
time). We multiply by $(r-1)$ because one ``bad'' point ``spoils'' the
vertex. Thus, \eqref{chernoff_upper} implies that
\[
\Pr(X_{\ref{A:b}}\geq |A_i|/20) \le \Pr(Y_{\ref{A:b}}\geq |A_i|/20)\leq \bfrac{40er(r-1)^2|A_i|}{n}^{|A_i|/20}=o(n^{-3}).
\]
We next observe that $X_{\ref{A:c}}$ is dominated by 
\[
Y_{\ref{A:c}} \sim (r-1)\Bin\brac{r|A_i|, \frac{4\log_{r-1}n}{q}}.
\] 
To see this we first observe that
$|E(w)|\leq 2\log_{r-1}n$, with room to spare. 
Consider an edge $e=(v_j,w)$ and
condition on the colors of every edge other than $e$. We examine the
effect of this conditioning, which we refer to as $\cC$.

We let $c(e)$ denote the color of edge $e$ in a given coloring. 
To prove our assertion about binomial domination, we prove that
for any color $x$,
\beq{colcond}
\Pr(c(e)=x\mid\cC)\leq \frac{2}{q}.
\eeq

We observe first that for a particular coloring $c_1,c_2,\ldots,c_m$ 
of the edges $e_1,e_2,\ldots,e_m$ we have
$$\Pr(c(e_i)=c_i,\,i=1,2,\ldots,m)=\prod_{i=1}^m\frac{1}{a_i}$$
where $q-\D\leq a_i\leq q$ is the number of 
colors available for the color of the edge $e_i$ 
given the coloring so far i.e. the number of colors
unused by the neighbors of $e_i$ in $\G$ when it is about to be colored.

Now fix an edge $e=e_i$ and the colors $c_j,\,j\neq i$. 
Let $C$ be the set of colors not used by the neighbors of $e_i$ in $\G$.
The choice by $e_i$ of its color under this conditioning 
is not quite random, but close. Indeed, we claim that for $c,c'\in C$
$$\frac{\Pr(c(e)=c\mid c(e_j)=c_j,\,j\neq i)}{\Pr(c(e)=
c'\mid c(e_j)=c_j,\,j\neq i)}\leq \bfrac{q-\D}{q-\D-1}^\D.$$
This is because, changing the color of $e$ only affects 
the number of colors available to neighbors of $e_i$, and only by at most one.
Thus, for $c\in C$, we have
\beq{uppcol}
\Pr(c(e)=c\mid c(e_j)=c_j,\,j\neq i)\leq \frac{1}{q-\D}\bfrac{q-\D}{q-\D-1}^\D.
\eeq
Now from \eqref{maxH} and \eqref{qcol} we see that
$\D\leq \frac{q}{K_1r}$ and so \eqref{uppcol} implies \eqref{colcond}.

Applying \eqref{chernoff_upper} we now see that
\[
\Pr(X_{\ref{A:c}}\geq |A_i|/20) \le \Pr(Y_{\ref{A:c}}\geq |A_i|/20) \leq \bfrac{80e(r-1)}{K_1^2}^{|A_i|/20}=o(n^{-3}).
\]
This completes the proof of \eqref{grow}. Thus, \eqref{grow} and \eqref{eq:Abounds} implies that \whp
\[
|A_{t_0}| \ge (r-1.1)^{t_0} \ge (r-1)^{\frac{1}{2}t_0} = n^{{1}/{20}}
\]
and
\[
|A_{t_0}| \le (r-1)^{t_0} |A_0| \le K_1 n^{{1}/{10}}\log n,
\]
since trivially $|A_0| \le K_1\log n$.

In a similar way, we define a sequence of sets $B_0=S_y,B_1,\ldots,B_{t_0}$
disjoint from $A_{\leq t_0}$. Here $T_y\cup B_{\leq t_0}$ spans a tree $T_{y,t_0}$.
As we go along 
we keep an injection $f_i:B_i\to A_i$ for $0\leq i\leq t_0$.
Suppose that $v\in B_i$. If $f_i(v)$ has no neighbors in $A_{i+1}$ because 
\eqref{A:b} or \eqref{A:c} failed then we do not try to add its neighbors to $B_{i+1}$.
Otherwise, we pair up its $(r-1)$ neighbors $b_1,b_2,\ldots,b_{r-1}$
outside $A_{\leq i}$ in
an arbitrary manner with the $(r-1)$ neighbors $a_1,a_2,\ldots,a_{r-1}$.
We will add $b_1,b_2,\ldots,b_{r-1}$ to $B_{i+1}$ and define 
$f_{i+1}(b_j)=a_j,\,j=1,2,\ldots,r-1$ if for each $1\leq j\leq r-1$ we have 
$b_j\notin A_{\leq t_0}\cup T_x\cup T_y\cup B_{\leq i+1}$ and 
the unique
path $P(b_j,y)$ of length $i+k_r$ from $b_i$ to $y$ in $T_{y,i}$
is rainbow colored and
furthermore, its colors are disjoint from the colors in the path
$P(a_j,x)$ in $T_{x,i}$. Otherwise, we do not grow from $v$.
The argument that we used for \eqref{grow} will show that
\[
\Pr\left(|B_{j+1}|\leq (r-1.1)|B_j| \; \big{|}\, K_2\log n\leq |B_j|\leq n^{2/3}\right)=o(n^{-3}).
\]
The upshot is that \whp\  we have $B_{t_0}$ and $A'_{t_0}=f_{t_0}(B_{t_0})$ of size
at least $n^{1/20}$. 

Our aim now is to show that \whp\  one can find vertex disjoint paths
of length $O(\log_{r-1}n)$ joining $u\in B_{t_0}$ to $f_{t_0}(u)\in A_{t_0}$
for at least half of the choices for $u$.

Suppose then that $B_{t_0}=\set{u_1,u_2,\ldots,u_p}$ and we have found 
vertex disjoint paths $Q_j$ joining $u_j$ and $v_j=f_{t_0}(u_j)$ for $1\leq j<i$.
Then we will try to grow breadth first trees $T_i,T_i'$ from $u_i$ and $v_i$ until
we can be almost sure of finding an edge joining their leaves. We will consider
the colors of edges once we have found enough paths. 

Let $R=A_{\leq t_0}\cup B_{\leq t_0}\cup T_x\cup T_y$. Then fix $i$ and define a sequence
of sets $S_0=\set{u_i},S_1,S_2,\ldots,S_t$ where we stop when either
$S_t=\emptyset$ or $|S_t|$ first reaches size $n^{3/5}$. Here $S_{j+1}=
N(S_j)\setminus (R\cup S_{\leq j})$. ($N(S)$ will be the set of neighbors
of $S$ that are not in $S$). The number of vertices excluded from $S_{j+1}$ is
less than $O(n^{1/10}\log n)$ (for $R$) plus $O(n^{1/10}\log n \cdot n^{3/5})$ for $S_{\leq j}$.
Since
\[
\frac{O(n^{1/10}\log n \cdot n^{3/5})}{n} = O(n^{-3/10}\log n) = O(n^{-3/11}),
\]
$|S_{j+1}|$ dominates the binomial random variable
\[
Z \sim \Bin\left((r-1)|S_j|,1-O(n^{-3/11})\right).
\]
Thus, by \eqref{chernoff_lower}
\begin{multline}\label{AT1}
\Pr\big{(}|S_{j+1}|\leq (r-1.1)|S_j|\; \big{|}\,100<|S_j|\leq n^{3/5}\big{)}\notag \\
\le \Pr\left(Z \leq (r-1.1)|S_j|\; \big{|}\, 100<|S_j|\leq n^{3/5}\right)\notag 
=o(n^{-3}).
\end{multline}
Therefore \whp, $|S_j|$ will grow at a rate $(r-1.1)$ once it reaches a
size exceeding~100. We must therefore estimate the number of times that 
this size is not reached. We can bound this as follows. If $S_j$ never
reaches 100 in size then some time in the construction of
the first $\log_{r-1}100$ $S_j$'s there will be an edge discovered between an
$S_j$ and an excluded vertex. The probability of this can be bounded by
$100 \cdot O(n^{-3/11})=O(n^{-3/11})$. So, if $\b$ denotes the number
of $i$ that fail to produce $S_t$ of size $n^{3/5}$ then
\[
\Pr(\b\geq 20)\leq o(n^{-3})+\binom{n^{1/10}\log n}{20} \cdot O(n^{-3/11})^{20}=o(n^{-3}).
\] 
Thus \whp\  there will be at least $n^{1/20}-20>n^{1/21}$ of the $u_i$ from which we
can grow a tree with $n^{3/5}$ leaves $L_{i,y}$
such that all these trees are
vertex disjoint from each other and $R$.

By the same argument we can find at least $n^{1/21}$ of the $v_i$ 
from which we
can grow a tree $L_{i,x}$ with $n^{3/5}$ leaves such that all these trees are
vertex disjoint from each other and $R$ {\em and the trees grown from the}
$u_i$. We then observe that if $e(L_{i,x},L_{i,y})$ denotes the edges from
$L_{i,x}$ to $L_{i,y}$ then
$$\Pr(\exists i:e(L_{i,x},L_{i,y})=\emptyset)\leq n^{1/20}\brac{1-\frac{(r-1)n^{3/5}}
{rn/2}}^{(r-1)n^{3/5}}=o(n^{-3}).$$
We can therefore \whp\  choose an edge $f_i\in e(L_{i,x},L_{i,y})$ for 
$1\leq i\leq n^{1/21}$. Each edge $f_i$ defines a path $Q_i$ from $x$ to $y$ of
length at most $2\log_{r-1}n$. Let $Q_i'$ denote that part of $Q_i$ that goes from
$u_i\in A_{t_0}$ to $v_i\in B_{t_0}$. The path $Q_i$ will be rainbow colored if the edges of
$Q_i'$ are rainbow colored and distinct from the colors in the path from
$x$ to $u_i$ in $T_{x,t_0}$ and the colors in the path from $y$ to $v_i$ in $T_{y,t_0}$. The probability that $Q_i'$ satisfies this condition is at least $\brac{1-\frac{2\log_{r-1}n}{q}}^{2\log_{r-1}n}$. Here we have used \eqref{colcond}. In fact, using \eqref{colcond} we see that 
\begin{align*}
\Pr(\not\exists i:Q_i\text{ is rainbow colored})&\leq 
\brac{1-\brac{1-\frac{2\log_{r-1}n}{q}}^{2\log_{r-1}n}}^{n^{1/21}}\\
&\leq \brac{1-\frac{1}{n^{4/(rK_1^2)}}}^{n^{1/21}}=o(n^{-3}).
\end{align*}
This completes the case where $x,y$ are both tree-like and $T_x\cap T_y=\emptyset$.
\subsubsection{Tree-like but not disjoint}
Suppose now that $x,y$ are both tree-like and $T_x\cap T_y\neq\emptyset$.
If $x\in T_y$ or $y\in T_x$ then there is nothing more to do as each root to
leaf path of $T_x$ or $T_y$ is rainbow.

Let $a\in T_y\cap T_x$ be such that its parent in $T_x$ is not in $T_y$.
Then $a$ must be a leaf of $T_y$. We now bound the number of leaves $\l_a$ in $T_y$
that are descendants of $a$ in $T_x$. For this we need the distance of $y$
from $T_x$. Suppose that this is $h$. Then 
$$\l_a=1+(r-2)+(r-1)(r-2)+(r-1)^2(r-2)+\cdots+(r-1)^{k_r-h-1}(r-2)=(r-1)^{k_r-h}+1.$$ 
Now from Lemma \ref{density} we see that there will be at most two choices
for $a$. Otherwise, $T_x\cup T_y$ will contain at least two cycles of length
less than $2k_r$. It follows that \whp\  there at most $\l_0=2((r-1)^{k_r-h}+1)$ leaves
of $T_y$ that are in $T_x$. If $(r-1)^h\geq 201$ then $\l_0\leq |S_y|/10$. Similarly,
if $(r-1)^h\geq 201$ then at most $|S_x|/10$ leaves of $T_x$ will be in $T_y$.
In which case we can use the proof for $T_x\cap T_y=\emptyset$ with $S_x,S_y$
cut down by a factor of at most $4/5$.

If $(r-1)^h\leq 200$, implying that $h\leq 5$ then we proceed as follows:
We just replace $k_r$ by $k_r+5$ in our definition of $T_x,T_y$, for these pairs. Nothing 
much will change. We will need to make $q$ bigger by a constant factor, 
but now we will have $y\in T_x$ and we are done.
\subsubsection{Non tree-like}\label{nontree}
We can assume that if $x$ is non tree-like then $T_x$ contains exactly one cycle $C$.
We first consider the case where $C$ contains an edge $e$ that is more than distance 5
away from $x$. Let $e=(u,v)$ where $u$ is the parent of $v$ and $u$ is at distance 5 from $x$. Let $\wT_x$ be obtained from $T_x$ by deleting the edge $e$ and adding two trees $H_u,H_v$, one rooted
at $u$ and one rooted at $v$ so that $\wT_x$ is a complete $(r-1)$-ary tree of height
$k_r$. Now color $H_u,H_v$ so that Lemma \ref{lemcol1} can be applied. 
We create $\wT_y$ from $T_y$ in the same way, if necessary.
We obtain
at least $(r-1)^{2k_r}/5$ pairs. But now we must subtract pairs that correspond to
leaves of $H_u,H_v$. By construction there are at most $4(r-1)^{2k_r-5}\leq (r-1)^{2k_r}/10$.
So, at least $(r-1)^{2k_r}/10$ pairs can be used to complete the rest of the proof as before.

We finally deal with those $T_x$ containing a cycle of length 10 or less,
no edge of which is further than distance 10 from $x$. 
Now the expected number of vertices on cycles of length $k\leq 10$ is given by
$$k\binom{n}{k}\frac{(k-1)!}{2}\binom{r}{2}^k2^k\frac{\Psi(rn-2k)}{\Psi(rn)}
\sim \frac{(r-1)^k}{2k},$$
where $\Psi(m)=m!/(2^{m/2}(m/2)!)$. 

It follows that the expected number of edges $\m$ that are within 10 or less from a 
cycle of length 10 or less is bounded by a constant. Hence $\m=o(\log n)$ \whp\ 
and we can give each of these edges a distinct new color after the first round of coloring. 
Any rainbow colored set of edges will remain rainbow colored after this change.

Then to find a rainbow path beginning at $x$ we first take a rainbow path to some $x'$ that is distance 10 from $x$ and then seek a rainbow path from $x'$.
The path from $x$ to $x'$ will not cause a problem as the edges on this path
are unique to it.

\section{The case $d=3$}
An easy generalization of the example in Figure \ref{fig1} shows that Lemma \ref{lemcol} does not extend to binary trees. It indicates that $\k_{\ell,2}\leq 2^\ell$ and not $\Omega(2^{2\ell})$ as we would like. In this case we have not been able to prove Corollary \ref{cor1}.
\begin{figure}[h]
\includegraphics[width=\textwidth]{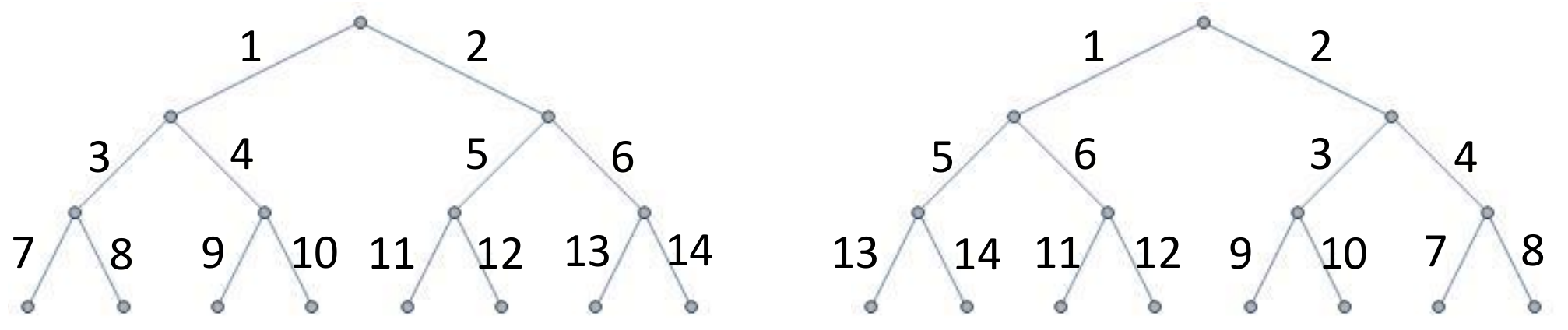}
\caption{Two rainbow trees $T_1$ and $T_2$ with $m(T_1,T_2) = 2^3$.}
\label{fig1}
\end{figure}
Note that while the example shows that $m(T_1,T_2)=2^\ell$, it does show there is a bijection $f$ between the leaves of the two trees so that $P_{x,T_1}\cup P_{f(x),T_2}$ is rainbow. In fact, an elegant probabilistic argument due to Noga Alon \cite{noga} shows that with the hypothesis of Lemma \ref{lemcol}, there are always sets $S_i\subseteq L_i$ and a bijection $f:S_1\to S_2$ such that (i) $|S_1|=|S_2|=\Omega(2^\ell)$ and such that (ii) $x\in S_1$ implies that $P_{x,T_1}\cup P_{f(x),T_2}$ is rainbow. This is a step in the right direction and it can be used to show that $O\brac{\bfrac{\log n}{\log\log n}^2}$ colors suffice, beating the bound implied by \cite{BCRR}.
\section{Conclusion} 
\label{sec:concl} 
We have shown that \whp\  $rc(G(n,r))=O(\log n)$ for $r\geq 4$ and 
$r=O(1)$. Determining the hidden constant seems challenging. We have seen that the argument for $d\geq 4$ cannot be extended to $d=3$ and so this case represents a challenge.

At a more technical level, we should also consider the case where $r\to\infty$ 
with~$n$. Part of this can be handled by the sandwiching results of Kim and Vu
\cite{KV} (see also~\cite{DFRS}).

\bigskip

\noindent
{\bf Acknowledgement} We are grateful to Noga Alon for help on the case $d=3$.

\providecommand{\bysame}{\leavevmode\hbox to3em{\hrulefill}\thinspace}

\end{document}